\documentclass[12pt]{amsart}

\usepackage{amsmath, amsthm, amssymb}
\usepackage{url} 
\usepackage{graphicx}
\usepackage{xcolor}
\usepackage{moreverb}

\usepackage{fancyvrb}

\def\r{\mathbb R}

\def\s{\mathbb S}

\def\e{\mathbb E}

\usepackage{listings}
 \lstnewenvironment{code}[1][]
  {\lstset{#1}}
  {}

\newtheorem{theorem}{Theorem}[section]
 \newtheorem{proposition}[theorem]{Proposition}
 \newtheorem{corollary}[theorem]{Corollary}
\theoremstyle{definition}
\newtheorem{definition}[theorem]{Definition}
\newtheorem{example}[theorem]{Example}
\newtheorem{remark}[theorem]{Remark}

\begin{document}

\title{Rectifying submanifolds of Riemannian manifolds with anti-torqued axis}

\author{Muhittin Evren Aydin$^1$}
\address{$^1$Department of Mathematics, Faculty of Science, Firat University, Elazig,  23200 Turkey}
\email{meaydin@firat.edu.tr}
\author{ Adela Mihai$^2$}
 \address{$^2$Technical University of Civil Engineering Bucharest,
Department of Mathematics and Computer Science, 020396, Bucharest, Romania
and Transilvania University of Bra\c{s}ov, Interdisciplinary Doctoral
School, 500036, Bra\c{s}ov, Romania}
 \email{adela.mihai@utcb.ro, adela.mihai@unitbv.ro}
\author{ Cihan Ozgur$^3$}
 \address{$^3$Department of Mathematics, Izmir Democracy University, 35140, Karabaglar, Izmir, Turkey}
 \email{cihan.ozgur@idu.edu.tr}

\begin{abstract}
In this paper we study rectifying submanifolds of a Riemannian manifold endowed with an anti-torqued vector field. For this, we first determine a necessary and sufficient condition for the ambient space to admit such a vector field. Then we characterize submanifolds for which an anti-torqued vector field is always assumed to be tangent or normal. A similar characterization is also done in the case of the torqued vector fields. Finally, we obtain that the rectifying submanifolds with anti-torqued axis are the warped products whose warping function is a first integration of the conformal scalar of the axis.
\end{abstract}
\maketitle

\section{Introduction} \label{intro}

In 2003, B.-Y. Chen \cite{crc0} initiated the study of the {\it rectifying curves}, a remarkable class of space curves in the Euclidean $3$-space $\e^3$, whose position vectors always lie in its rectifying plane, closely related to the famous space curves, e.g. geodesics, helices, centrodes and etc. Since then, it has been of great interest, and without making a complete list, we refer to \cite{cgb,crc1,crc2,dca,inpt,in,jamr,loy0,loy1}. 

The notion of a rectifying curve was extended to higher dimensions in \cite{crs0,crs1} by defining a {\it rectifying submanifold} in $\e^m$ as a submanifold for which the position vector always lie in its rectifying space, i.e. the orthogonal complement to the first normal space. Such submanifolds were completely classified in the cited works \cite{crs0,crs1}. The counterpart in the pseudo-Euclidean setting can be found in \cite{crs2}. Afterwards, such submanifolds were adapted to an arbitrary Riemannian manifold as follows. 

\begin{definition} \cite{crs3} \label{def-rs}
Let $V$ a non-zero vector field on a Riemannian $m$-manifold $(\widetilde{M},\tilde{g})$ and $M$ a submanifold of $\widetilde{M}$. Let also, for a point $p\in M$, $\text{Im } h_p$ be the first normal space of $M$ in $\widetilde{M}$. Then, $M$ is called a {\it rectifying submanifold} in $\widetilde{M}$ with respect to $V$ if the normal component $V^{\perp} \neq 0$ on $M$ and 
\begin{equation} \label{rs-intro}
\tilde{g}(V(p),\text{Im } h_p)=0, \quad \forall p\in M.
\end{equation}
\end{definition}

We call the submanifold $M$ a {\it rectifying submanifold with axis} $V$. With this definition, the  study of rectifying submanifolds in Riemannian manifolds is an interesting problem, especially under certain conditions on the axis $V$. In order to clarify it, let $\widetilde{\nabla}$ be the Levi-Civita connection on $\widetilde{M}$. A vector field $V$ on $\widetilde{M}$ is called a {\it torse-forming} vector field \cite{ya0} if
\begin{equation}\label{tors}
\widetilde{\nabla}_X V=f X + \omega (X)V, \quad \forall X\in \Gamma(T\widetilde{M}),
\end{equation}
for a $1$-form $\omega$ and a smooth function $f$ on $\widetilde{M}$, where the $1$-form $\omega$ and the function $f$ are called the {\it generating form} and {\it conformal scalar} of $V$, respectively \cite{mih0}. Particularly, it is called a {\it concircular} vector field if the generating form $\omega$ vanishes identically. Such vector fields have applications in Physics, see \cite{c0,dr,mam0,mam1}.

A subclass of torse-forming vector fields, called {\it torqued} vector fields, was introduced by imposing the condition $\omega(V) =0$ on $\tilde{M}$ in \eqref{tors} (see \cite{crs3}). Then, in the same paper \cite{crs3}, a necessary and sufficient condition for a Riemannian manifold (also Lorentzian manifold) to admit a torqued vector field was determined. In addition, some characterizations of rectifying submanifolds in a Riemannian manifold with torqued axis were given in \cite{crs3}.

Another subclass of torse-forming vector fields was introduced by S. Deshmukh et al. \cite{dan}, so-called {\it anti-torqued} vector fields: let $V$ be a torse-forming vector field given by \eqref{tors} and $W$ dual to $\omega$. Particularly, assume $W=-fV$, and so $\omega(X)=-f\nu(X)$, where $\nu$ is the dual of $V$ and $X\in T\widetilde{M}$. Notice here that, in the case of torqued vector fields, $W$ is perpendicular to $V$. An anti-torqued vector field is a non-vanishing vector field $V$ on $\widetilde{M}$ fulfilling
\begin{equation}\label{a-tor}
\widetilde{\nabla}_X V=f (X - \nu (X)V), \quad \forall X\in \Gamma(T\widetilde{M}).
\end{equation}

There are important differences between torqued and anti-torqued vector fields. For example it can be seen from \eqref{a-tor} that a unit anti-torqued vector field is also a unit geodesic vector field, which is not possible for the torqued ones. Also an anti-torqued vector field cannot be reduced to a concircular vector field, in constrast to a torqued vector field. Moreover, there is no a torqued vector field that is also a gradient vector field (see \cite[Proposition 3.1]{crs3}), while there are examples for the other case \cite{dan}. With the features, the anti-torqued vector fields generate an independent subclass of the torse-forming vector fields on Riemannian manifolds.

The global existence of a nonzero vector field on a Riemannian manifold is a tool to better understand its geometry and topology. For instance, various characterizations of spheres and Euclidean spaces in \cite{dk,dan,dta,dmtv,id} were done through certain vector fields, e.g. a geodesic vector field, a Killling vector field, and a torse-forming vector field. In particular, while on a sphere there are no globally defined anti-torqued vector fields (except the trivial ones), on a Euclidean space they can be globally defined \cite{dan}.

This paper is organized as follows. In Sect. \ref{pre}, we first give the fundamental formulas and equations in submanifold theory and recall some well-known equalities for warped products to use in our calculations. In Sect. \ref{anti-warp}, we state that a Riemannian manifold admits an anti-torqued vector field if and only if it is locally a warped product (Thm. \ref{nec-suf}), which is similar to a consequence in \cite{crs3} that it is locally a twisted product in the case of the torqued vector fields. After that we characterize submanifolds, in Sect. \ref{charac-sub}, on which the tangential or normal components of a torqued or an anti-torqued vector field vanish identically (Propositions \ref{anti-tan0-normal0} and \ref{torqued-normal0}). Several examples are provided. Sect. \ref{rec-sub-anti} is devoted to classify rectifying submanifolds when the axis is assumed to be anti-torqued. Clearly, we obtain that a rectifying submanifold of a Riemannian manifold with axis $V$ is a warped product (see Thm. \ref{rect-sub-anti}) such that the warping function is a first integration of the conformal scalar of $V$. In the Euclidean ambient space, the examples for such submanifolds are also given.

\section{Preliminaries} \label{pre}

Let $(\widetilde{M},\tilde{g})$ be a Riemannian $m$-manifold and $\widetilde{\nabla}$ the Levi-Civita connection on $\widetilde{M}$. Let $\Gamma(T\widetilde{M})$ the set of the sections of the tangent bundle of $\widetilde{M}$. The Riemannian curvature tensor of $\widetilde{\nabla}$ is
$$
\widetilde{R}(X,Y)Z=\widetilde{\nabla}_X\widetilde{\nabla}_YZ-\widetilde{\nabla}_Y\widetilde{\nabla}_XZ-\widetilde{\nabla}_{[X,Y]}Z, \quad X,Y,Z \in \Gamma(T\widetilde{M}).
$$
Let $u,v\in T_p\widetilde{M}$ be two linearly independenty vectors tangent to $\widetilde{M}$ at $p \in \widetilde{M}$ and $\pi=\text{Span}\{u,v\}$ a plane section. Then, the sectional curvature is defined by
$$
\widetilde{K}(\pi)=\widetilde{K}(u,v)=\frac{\tilde{g}(\widetilde{R}(u,v)v,u)}{\tilde{g}(u,u)\tilde{g}(v,v)-\tilde{g}(u,v)^2}.
$$
We call $\widetilde{M}$ a {\it real space form} if the sectional curvature $\widetilde{K}$ is constant, say $c$, for every plane section $\pi$ and every point $x \in \widetilde{M}$. In this case, the curvature tensor is given by
$$
\widetilde{R}(X,Y)Z=c\{\tilde{g}(Y,Z)X-\tilde{g}(X,Z)Y\}, \quad \forall X \in \Gamma(T\widetilde{M}).
$$

Let $M$ be an $n$-dimensional submanifold of $\widetilde{M}$, $T_pM$ and $N_pM$ the tangent and the normal spaces of $M$ in $\widetilde{M}$ at some point $p\in M$, respectively. A natural orthogonal decomposition is written by
$$
T_p\widetilde{M}=T_pM \oplus N_pM , \quad \forall p \in M.
$$
The formula of Gauss is given by
$$
\widetilde{\nabla}_XY=\nabla_XY+h(X,Y), \quad \forall X,Y \in\Gamma(TM),
$$
where $\nabla$ is the induced Levi-Civita connection on $M$ and $h$ the second fundamental form. A {\it totally geodesic submanifold} is defined by the condition $h=0$. The {\it first normal space} of $M$ at some point $p\in M$ is a subspace of $N_pM$ defined by
$$
\text{Im }h_p =\text{Span}\{h(X,Y): X,Y \in T_pM \}.
$$

The formula of Weingarten is given by
$$
\tilde{\nabla}_X\xi=-A_{\xi}(X)+D_X\xi , \quad \forall X \in\Gamma(TM), \quad \forall \xi \in \Gamma(NM),
$$
where $A$ is the shape operator of $M$ and $D$ is the normal connection in $\Gamma(NM)$. A  relationship between $h$ and $A$ is given by
$$
\tilde{g}(h(X,Y),\xi)=g(A_{\xi}X,Y), \quad \forall X,Y \in \Gamma(TM), \quad \forall \xi \in \Gamma(NM),
$$
where $g$ is the induced metric tensor of $M$.

Let $\{e_1,...,e_n\}$ be an orthonormal frame on $M$. The {\it mean curvature} is defined by
$$
H=\frac{1}{n}\sum_{i=1}^nh(e_i,e_i).
$$
We call $M$ {\it umbilical} with respect to a vector field $\xi$ normal to $M$ if $A_{\xi}=\mu\text{Id}$, for a smooth function $\mu$ on $M$. It is also called {\it totally umbilical} if it is umbilical with respect to every normal direction. 

From the formulas of Gauss and Weingarten, it follows
\begin{equation}\label{gauss-eq1}
\left. 
\begin{array}{c}
\tilde{R}(X,Y)Z= R(X,Y)Z-A_{h(Y,Z)}X+A_{h(X,Z)}Y+h(X,\nabla_YZ) \\
-h(Y,\nabla_XZ)-h([X,Y],Z)+D_Xh(Y,Z)-D_Yh(X,Z)
\end{array}%
\right. 
\end{equation}
and then
\begin{equation}\label{gauss-eq2}
\left. 
\begin{array}{c}
g(R(X,Y)Z,W)=\tilde{g}(\widetilde{R}(X,Y)Z,W)+\tilde{g}(h(X,W),h(Y,Z))\\-\tilde{g}(h(X,Z),h(Y,W)).
\end{array}%
\right. 
\end{equation}
This is known as the {\it Gauss equation}.

\subsection{Warped products}

Let $(M_1,g_1)$ and $(M_2,g_2)$ be two Riemannian manifolds and consider the natural projections of the product $M_1\times M_2$ 
$$
\pi_i : M_1 \times M_2 \to M_i, \quad i=1,2.
$$
A {\it twisted product} $\widetilde{M}=M_1 \times_{\lambda} M_2$ is the manifold $M_1 \times M_2$ equipped with the metric tensor
$$
\tilde{g}=g_1+\lambda^2g_2,
$$
where $\lambda$ is a smooth function on $\widetilde{M}$. Particularly, if the function $\lambda$ depends only on $M_1$ then $\widetilde{M}$ is called a {\it warped product} and $\lambda$ {\it warping function}. $M_1$ and $M_2$ are said to be the base and the fiber of $\widetilde{M}$, respectively.

The leaves $M_1 \times \{ q\} = \pi_2^{-1}(q)$ and the fibers $\{p \} \times M_2 = \pi_1^{-1}(p)$ are submanifolds of $\widetilde{M}$. We call a vector tangent to leaves (resp. fibers) {\it horizontal} (resp. {\it vertical}). Denote by $\mathcal{L}(M_1)$ (resp. $\mathcal{L}(M_2)$) the set of all horizontal (resp. vertical) lifts.

A useful result that we will use later is the following.

\begin{proposition} \cite{cbook} \label{warp-con}
Let $\widetilde{\nabla}$ and $^i\widetilde{\nabla}$ be the Levi-Civita connections on $\widetilde{M}=M_1 \times_{\lambda} M_2$ and $M_i$ $(i=1,2)$, respectively. Let also $X_i,Y_i \in \mathcal{L}(M_i)$. Then,
\begin{enumerate}
\item $\widetilde{\nabla}_{X_1}Y_1 \in \mathcal{L}(M_1) $ is the lift of $^1\widetilde{\nabla}_{X_1}Y_1$,
\item $\widetilde{\nabla}_{X_1}X_2 =\widetilde{\nabla}_{X_2}X_1 =(X_1 \log \lambda)X_2$.
\end{enumerate}
\end{proposition}




\section{Anti-Torqued Vector Fields and Warped Products} \label{anti-warp}

Recall that a Riemannian $m$-manifold $\widetilde{M}$ admits a torqued vector field if and only if it is locally the twisted product $I\times_{\lambda}F$, where $F$ is a Riemannian $(m-1)$-manifold and $I$ is an open real interval (see \cite[Theorem 3.1]{crs3}). In this section, we consider the case that $\widetilde{M}$ admits an anti-torqued vector field.

\begin{theorem} \label{nec-suf}
Let $F$ be a Riemannian $(m-1)$-manifold, $I$ an open real interval and $\lambda$ a smooth function on $I$. Then, a Riemannian $m$-manifold admits an anti-torqued vector field if and only if it is locally the warped product $I\times_{\lambda}F$.
\end{theorem}

\begin{proof}
Let $(\widetilde{M},\tilde{g})$ be a Riemannian $m$-manifold endowed with an anti-torqued vector field $V$. We will show that $\widetilde{M}$ is locally a warped product. If $\nu$ is the dual of $V$, then $\nu(X)=\tilde{g}(X,V)$, for every $X \in T\widetilde{M}$ and so \eqref{a-tor} writes as
\begin{equation} \label{nec-suf-atorq}
\tilde{\nabla}_XV=f(X-\tilde{g}(X,V)V),  \quad \forall X \in T\widetilde{M}.
\end{equation}
Denote by $E_1\in \Gamma(T\widetilde{M})$ a unit vector field. There is  a smooth function $\lambda >0$ on $\widetilde{M}$ such that $V=\lambda E_1$. The covariant differentiation of $V$ is
$$
\widetilde{\nabla}_VV=\lambda E_1 (\lambda)E_1+ \lambda^2 \widetilde{\nabla}_{E_1}E_1,
$$
and, by \eqref{nec-suf-atorq},
$$
\widetilde{\nabla}_VV=f\lambda (1- \lambda^2)E_1.
$$
Because $\lambda \neq 0$ and $E_1 \perp \widetilde{\nabla}_{E_1}E_1$, we conclude 
$$
E_1(\lambda)=f(1-\lambda^2), \quad \widetilde{\nabla}_{E_1}E_1=0,
$$
implying the integral curves of $E_1$ are geodesics in $\widetilde{M}$. Setting $\mathcal{D}=\text{Span}\{ E_1\}$, we have an integrable distribution $\mathcal{D}$ in which the leaves are totally geodesic in $\widetilde{M}$.

We now extend $E_1$ to a local orthonormal frame $\{E_1,...,E_m\}$ on $\widetilde{M}$. The connection forms are
$$
\widetilde{\nabla}_{E_j}E_i=\sum_{k=1}^m \omega_i^k(E_j) E_k , \quad \omega_i^k+\omega_k^i=0. \quad i,j \in \{1,...,m \}.
$$
Set $\mathcal{D}^{\perp}=\text{Span}\{E_2,...,E_n \}.$ Because $\tilde{g}(E_j,E_1)=0$, for $j =2,...,m$, \eqref{nec-suf-atorq} is now
\begin{equation} \label{ejv-1}
\tilde{\nabla}_{E_j}V =fE_j, \quad j=2,...,m.
\end{equation}
The covariant differentiation of $V$ is
\begin{equation} \label{ejv-2}
\widetilde{\nabla}_{E_j}V =E_j(\lambda)E_1+\lambda\sum_{k=1}^m \omega_1^k(E_j) E_k,  \quad j=2,...,m.
\end{equation}
It follows from Eqs. \eqref{ejv-1} and \eqref{ejv-2} that
\begin{equation} \label{delta}
\omega_1^k(E_j)=\frac{f}{\lambda} \delta_{jk}, \quad \forall j,k \in \{2,...,m \}, 
\end{equation}
and
\begin{equation*} \label{warped}
E_j(\lambda)=0, \quad j =2,...,m .
\end{equation*}
It is obvious that 
$$
\omega_1^k(E_j)=\omega_1^j(E_k), \quad \forall j,k \in \{2,...,m \}
$$
and so the integrability of the distribution $\mathcal{D}^{\perp}$ is guaranteed by the Frobenius Theorem. In addition, the leaves of $\mathcal{D}^{\perp}$ are totally umbilical hypersurfaces in $\widetilde{M}$ whose mean curvature is $f/\lambda$. Since those leaves are hypersurfaces, their mean curvature vector fields become parallel, namely $\mathcal{D}^{\perp}$ is a spheric foliation. Then, by \cite[Proposition 3]{pr}, $\widetilde{M}$ is locally the warped product $I \times_{\lambda}F$, where $I$ is an open interval, $F$ is a Riemannian $(m-1)$-manifold, and $\lambda >0$ a smooth function on $I $. Hence, the metric tensor $\tilde{g}$ of $\widetilde{M}$ becomes
\begin{equation}\label{tmetric}
\tilde{g} = ds^2 +\lambda^2 g, \quad E_1=\frac{\partial}{\partial s}, 
\end{equation}
where $g$ denotes the metric tensor of $F$. Moreover, the anti-torqued form $\nu=\lambda ds$, where $ds$ is the dual of the lift to $\widetilde{M}$ of $\frac{\partial}{\partial s}$.

Conversely, assume that $\widetilde{M}$ is a non-trivial warped product $I\times_{\lambda}F$ endowed with the metric tensor as in \eqref{tmetric}. By Proposition \ref{warp-con}, we have
$$
\widetilde{\nabla}_{\frac{\partial}{\partial s}}\frac{\partial}{\partial s} =0 
$$
and 
$$
\widetilde{\nabla}_{\frac{\partial}{\partial s}}Y =\widetilde{\nabla}_{Y}\frac{\partial}{\partial s} =\left ( \frac{d \log \lambda}{d s} \right )Y,
$$
for every $Y \in \mathcal{L}(F)$. Define $V= \frac{\partial}{\partial s}$.  We write
\begin{equation}\label{warp-XV0}
\widetilde{\nabla}_XV=\left\{ 
\begin{array}{ll}
0, & X\parallel V, \\ 
\frac{d \log \lambda}{d s}X, & X\perp V,%
\end{array}%
\right. 
\end{equation}
for every $X$ tangent to $\widetilde{M}$. Assume now that $\nu$ is a $1-$form on $\widetilde{M}$ which is the dual of the lift to $\widetilde{M}$ of $\frac{\partial}{\partial s}$. Then, we have $\nu  (Y)=\tilde{g}(Y,V)=0$, for every $Y \in \mathcal{L}(F)$. Hence, we conclude
\begin{equation}\label{warp-XV1}
f(X-\nu  (X)V)=\left\{ 
\begin{array}{ll}
0, & X\parallel V, \\ 
fX, & X\perp V.%
\end{array}%
\right. 
\end{equation}
If we set $f=\frac{d \log \lambda}{ds}$ and if we compare Eqs. \eqref{warp-XV0} and \eqref{warp-XV1}, then we conclude that $V= \frac{\partial}{\partial s}$ is an anti-torqued vector field on $\widetilde{M}$, completing the proof.
\end{proof}

\begin{remark}
By Eqs. \eqref{warp-XV0} and \eqref{warp-XV1}, we point out that $V=\frac{\partial}{\partial s}$ cannot be a concircular vector field because $\widetilde{\nabla}_VV=0$.
\end{remark}

\section{Characterizations of Submanifolds by (Anti-)Torqued Vector Fields} \label{charac-sub}

Let $\widetilde{M}$ be a Riemannian manifold endowed with a certain nonzero vector field $V$. In case $V$ is an anti-torqued or a torqued vector field, we characterize submanifolds of $\widetilde{M}$ depending on whether the tangential or the normal components of $V$ with respect to $M$ are identically zero. We conduct this investigation in separate cases, as there are significant differences between such vector fields, which have already been mentioned in Sect. \ref{intro}.

\subsection{Characterizations via anti-torqued vector fields} \label{charac-sub-antitorq}
We point out that without loss of generality, an anti-torqued vector field can be assumed to be unit (see \cite{dan}). We prove:

\begin{theorem} \label{anti-tan0-normal0}
Let $\widetilde{M}$ be a Riemannian manifold endowed with a unit anti-torqued vector field $V$, $M$ a submanifold of $\widetilde{M}$, and $A$ the shape operator of $M$.
\begin{enumerate}

\item If the tangential component $V^{\top}=0$ identically on $M$, then $V^{\perp}$ is parallel and umbilic direction.

\item If the normal component $V^{\perp}=0$ identically on $M$, then  $|A_{\xi}|=0$, for every $\xi \in \Gamma(NM)$ and
$$
\tilde{R}(X,Y)V^{\top}=R(X,Y)V^{\top}, \quad \forall X,Y \in \Gamma(TM) .
$$
\end{enumerate}
\end{theorem}

\begin{proof}
Let $\widetilde{\nabla}$ be the Levi-Civita connection on $\widetilde{M}$. We distinguish the two cases:

{\it Case 1.} Assume that tangential component $V^{\top}=0$ on $M$. By the definition of an anti-torqued vector field and the formula of Weingarten, we have
$$
fX=\widetilde{\nabla}_{X} V^{\perp}=-A_{V^{\perp}}(X)+ D_X V^{\perp} , \quad \forall X\in \Gamma(TM),
$$
which implies
$$
D_X V^{\perp}=0, \quad A_{V^{\perp}}(X)=-fX, \quad \forall X \in \Gamma(TM).
$$
Then we gave the proof of the first item.

{\it Case 2.}  If the normal component $V^{\perp}=0$ on $M$, then by using the similar argument as in the first case, we get
\begin{equation*} 
f(X-  g(X,V^{\top})V^{\top})=\widetilde{\nabla}_{X} V^{\top}=\nabla_XV^{\top}+ h(X,V^{\top})  , \quad \forall X\in \Gamma(TM).
\end{equation*}
Here the normal component is identically $0$, namely
\begin{equation} \label{anti-case2-h2}
h(X,V^{\top})=0, \quad \forall X \in \Gamma(TM),
\end{equation}
or equivalently
$$
g(A_{\xi}(V^{\top}),X)=0, \quad \forall X \in \Gamma(TM), \quad \forall \xi \in \Gamma(NM),
$$
which implies that $A_{\xi}$ has determinant $0$. We next substitute Eq. \eqref{anti-case2-h2} in Eq. \eqref{gauss-eq1}, obtaining
$$
\widetilde{R}(X,Y)V^{\top}=R(X,Y)V^{\top}+h(X,\nabla_YV^{\top})-h(Y,\nabla_XV^{\top}), \quad \forall X,Y \in \Gamma(TM).
$$
By the definition of an anti-torqued vector field, we get
\begin{eqnarray*}
\widetilde{R}(X,Y)V^{\top}&=&R(X,Y)V^{\top}+f(h(X,Y)-g(Y,V^{\top})h(X,V^{\top})) \\ 
&-&f(h(Y,X)-g(X,V^{\top})h(Y,V^{\top})), \quad \forall X,Y \in \Gamma(TM).
\end{eqnarray*}
Considering Eq. \eqref{anti-case2-h2} in the above equation completes the proof.
\end{proof}

In particular cases, we indicate some consequences of Thm. \ref{anti-tan0-normal0} as well as examples. For this, let $\widetilde{M}$ be a Riemannian manifold endowed with a unit anti-torqued vector field $V$ and $M$ be a hypersurface of $\widetilde{M}$. If the tangential component $V^{\top}=0$ identically on $M$, then according to the first item of Proposition \ref{anti-tan0-normal0} it is totally umbilical in $\widetilde{M}$. When the ambient space $\widetilde{M}$ is the Euclidean space, the hyperspheres are the only hypersurfaces with $V^{\top}=0$. 

\begin{corollary} \label{anti-tot-um}
Let $\widetilde{M}$ be a Riemannian manifold endowed with an anti-torqued vector field $V$. Then, the totally umbilical hypersurfaces are the only hypersurfaces in $\widetilde{M}$ for which the tangential component of $V$ in $M$ is identically zero.
\end{corollary}

In the following we give an example for a submanifold of codimension $2$ fulfilling the first item of Thm. \ref{anti-tan0-normal0}.

\begin{example}[$V^{\top}=0$]
Let $(x_1,x_2,x_3,x_4)$ be the canonical coordinates in $\e^4$ and $M=\mathbb{S}^1(r_1)\times \mathbb{S}^1(r_2)$ be the Clifford torus in $\e^4$ with $r_1^2+r_2^2=1$. Let also $E$ be the radial vector field in $\e^4$. Then, $V=E/|E|$ is a unit anti-torqued vector field on the connected Riemannian manifold $\e^4-\{0\}$ \cite{dan}. Let $\widetilde{\nabla}$ be the Levi-Civita connection on $\e^4$. Then,
$$
\widetilde{\nabla}_{\widetilde{X}}V=\frac{1}{|E|}(\widetilde{X}-\langle \widetilde{X},V\rangle V), 
$$
for every $\widetilde{X}$ tangent to $\e^4$. Here the conformal scalar $f$ of $V$ is
$$
f(x_1,...,x_4)=\left (\sum_{i=1}^4 x_i^2 \right )^{-1/2}, \quad (x_1,...,x_4) \in \e^4-\{0\}.
$$
Clearly, $V$ is normal to $M$. As the restriction $V|_M$ is radial, we have
$$
X=\widetilde{\nabla}_XV=-A_V(X)+D_XV, \quad \forall X\in\Gamma(TM),
$$
implying $V$ is parallel and $A_V=-X$, namely $V$ is an umbilic direction. 
\end{example}

Although the equality $V^{\perp}=0$ is not allowed in Definition \ref{def-rs}, it is naturally satisfied as long as $M$ is a hypersurface of $\widetilde{M}$. In such a case, by the second item of Theorem \ref{anti-tan0-normal0} the following consequence can be stated.

\begin{corollary} \label{anti-degenerate}
Let $\widetilde{M}$ be a Riemannian manifold endowed with an anti-torqued vector field $V$ and $M$ be a rectifying hypersurface with axis $V$. Then, the shape operator of $M$ has zero determinant. 
\end{corollary}

In the following we exhibit two examples of such hypersurfaces with $|A|=0$.

\begin{example}[$V^{\perp}=0$] \label{tang-dev}
Let $c(s)$ be a unit speed geodesic circle in $\s^2\subset\e^3$ and $M$ be a tangent developable parametrized by
$$
r(s,t)=c(s)+t c'(s), \quad s\in I\subset \r, \quad t\in \r.
$$
The unit normal vector field to $M$ is $-c(s)\times c'(s)$. If we take the anti-torqued vector field $V=E/|E|$ on $\e^3-\{0\}$, then we see $V^{\perp}=0$ on $M$. Clearly, the shape operator of every tangent developable in $\e^3$ has zero determinant (see \cite{gr}).
\end{example}

\begin{example}[$V^{\perp}=0$] \label{anti-cone}
Let $M$ be a cone in $\e^m$ with vertex at the origin given in implicit form
$$
\varphi(x_1,...,x_m)=x_1^2+...+x_{m-1}^2-x_m^2=0, \quad (x_1,...,x_m) \neq (0,...,0).
$$
Let also $V=E/|E|$ the unit anti-torqued vector field on $\e^m-\{0\}$. As $V|_M$ is perpendicular to $\text{grad } \varphi$, $V|_M$ is always tangent to $M$, where $\text{grad}$ is the gradient of $\e^m$.
\end{example}

Let $M$ be a submanifold of $\widetilde{M}$ and $V$ an anti-torqued vector field on $\widetilde{M}$. Assume that the normal component $V^{\perp}=0$ identically on $M$. By the second item of Thm. \ref{anti-tan0-normal0}, we have
$$
\tilde{g}(\tilde{R}(X,V^{\top})V^{\top},X)=g(R(X,V^{\top})V^{\top},X), \quad \forall X\in \Gamma(TM).
$$
If $\pi=\text{Span}\{X,V^{\top}\}$ is any plane section at $x\in M$, then the above equation yields
\begin{equation}\label{anti-sec}
\widetilde{K}(\pi)=K(\pi).
\end{equation}
Now we go back to Example \ref{tang-dev} and recall that any tangent developable has vanishing Gaussian curvature. In view of that the sectional curvature of a surface in $\e^3$ is equivalent to the Gaussian curvature, we see that \eqref{anti-sec} is clearly satisfied for the surface given in Example \ref{tang-dev}.

\subsection{Characterizations via torqued vector fields} \label{charac-sub-torq}

We characterize submanifolds through the torqued vector fields $V$ as we did in the previous subsection. Let $\omega$ and $f$ be the torqued form and conformal scalar of $V$. If we denote by $W$ the dual of $\omega$, then we know $\tilde{g}(V,W)=0$. The natural orthogonal decompositions of $V$ and $W$ in $M$ are given uniquely by
\begin{equation} \label{tor-v}
V=V^{\top}+ V^{\perp}, \quad W=W^{\top}+ W^{\perp}.
\end{equation}

\begin{proposition} \label{torqued-normal0}
Let $\widetilde{M}$ be a Riemannian manifold endowed with a torqued vector field $V$, $M$ a submanifold of $\widetilde{M}$, and $A$ the shape operator of $M$.
\begin{enumerate}

\item If $V^{\perp}=0$ identically on $M$, then $V^{\top}$ is a concircular vector field on $M$ and, for every $\xi \in \Gamma(NM)$, the shape operator $A_{\xi}$ has zero determinant.

\item If $V^{\top}=0$ identically on $M$, then $V^{\perp}$ is an umbilical direction and 
$$
D_XV^{\perp}=0, \quad D_{W^{\top}}V^{\perp}=|W^{\top}|^2V^{\perp},
$$
for every $X\perp W^{\top}$ tangent to $M$ and the normal connection $D$ on $\Gamma(NM)$.
\end{enumerate}
\end{proposition}

\begin{proof}
We distinguish two cases:
 
{\it Case 1.}  Assume that $V|_M\in\Gamma(TM)$. By \eqref{tor-v} we have $W|_M\in\Gamma(NM) $. As $V$ is a torqued vector field, we get 
$$
\widetilde{\nabla}_{X}V=fX+\tilde{g}(W,X)V=fX, \quad \forall X \in \Gamma (TM).
$$
From the formula of Gauss, we see $\nabla_{X}V=fX$, namely $V|_M$ is concircular on $M$. Also, we conclude $h(X,V)=0$, for every $X$ tangent to $M$ which is equivalent to $|A_{\xi}|=0$, for every $\xi$ normal to $M$.

{\it Case 2.} If $V|_M\in\Gamma(NM)$, then from \eqref{tor-v} it follows $W|_M\in\Gamma(NM)$. By the definition of a torqued vector field and the formula of Weingarten, we write
$$
-A_V(X) +D_XV=\widetilde{\nabla}_{X}V=f X+ g(W,X)V, \quad \forall X \in \Gamma (TM).
$$
Comparing the tangential and normal components we complete the proof.

\end{proof}

\section{Rectifying submanifolds with anti-torqued axis} \label{rec-sub-anti}

Let $\widetilde{M}$ be a Riemannian $m$-manifold and $M$ be an $n$-dimensional submanifold of $\widetilde{M}$. Let also $V$ be a torqued vector field on $\widetilde{M}$ with $V^{\top}\neq 0$ on $M$. Recall that $M$ is a rectifying submanifold with torqued axis $V$ if and only if $V^{\top}$ is a torse-forming vector field on $M$ such that the  conformal scalar and the generating form are the restrictions of the torqued function and the torqued form on $M$, respectively \cite[Theorem 4.1]{crs3}. 

Motivated by the above result we will classify the rectifying submanifolds $\widetilde{M}$ with anti-torqued axis. We say that $M$ is {\it proper} if $V^{\top}\neq 0$ and $V^{\perp}\neq 0$ on $M$. 

\begin{theorem} \label{rect-sub-anti}
Let $\widetilde{M}$ be a Riemannian $m$-manifold endowed with an anti-torqued vector field $V$ and $M$ be an $n$-dimensional proper rectifying submanifold with axis $V$. Assume that $1<n<m-1$. Then, $M$ is locally the warped product $J\times_{\lambda} N$, where $J$ is an open interval, $N$ is a Riemannian $(n-1)$-manifold and the metric tensor of $M$ is
\begin{equation} \label{anti-rec-met}
g=ds^2+\left (\tanh\int^sf(u)du \right )^2g_N, \quad \frac{V^{\top}}{|V^{\top}|}=\frac{\partial}{\partial s},
\end{equation}
where $f$ is the conformal scalar of $V$ and $g_N$ is the metric tensor of $N$.
\end{theorem}

\begin{proof}
By Definition \ref{def-rs}, we write
$$
\tilde{g}(V^{\perp},h(X,Y))=0, \quad \forall X,Y\in \Gamma(TM),
$$
which immediately yields $A_{V^{\perp}}=0$ on $M$. Also, from the definition of anti-torqued vector field it follows
\begin{equation}
\widetilde{\nabla}_X (V^{\top}+ V^{\perp}) = f\{X- g( X,V^{\top})(V^{\top}+ V^{\perp})\}, \label{XTN}
\end{equation}
where $\widetilde{\nabla}$ is the Levi-Civita connection on $\widetilde{M}$. Using the formulas of Gauss and Weingarten, one gets
\begin{equation}
\widetilde{\nabla}_{X} (V^{\top}+ V^{\perp}) = \nabla_X V^{\top}+ h(X ,V^{\top})+ D_X  V^{\perp} \label{XTN-2}.
\end{equation}
Comparing the tangential components in Eqs. \eqref{XTN} and \eqref{XTN-2}, we deduce
\begin{equation} \label{anti-rec-t}
fX=f g( X,V^{\top})V^{\top}+ \nabla_X V^{\top}.
\end{equation}
Let $E_1$ be a unit tangent vector field and $V^{\top}=\lambda E_1$, for a positive smooth function $\lambda$ on $M$. We may extend it to a local orthonormal frame $\{E_1,E_2,...,E_n\}$ on $M$. Writing $X=E_1$ in Eq. \eqref{anti-rec-t}, we have
\begin{equation}\label{anti-rec-e1}
\{(-1+\lambda^2)f +E_1(\lambda)\}E_1 +\lambda\nabla_{E_1} E_1=0,
\end{equation}
where $\nabla_{E_1} E_1=0$ because $E_1 \perp \nabla_{E_1} E_1$. Analogously, plugging $X=E_j$ ($j=2,...,n$) into Eq. \eqref{anti-rec-t}, we obtain 
$$
f E_j = E_j(\lambda)E_1+\lambda \nabla_{E_j} E_1.
$$ 
Taking inner product of this equation by $E_1$, we get
\begin{equation}\label{eje1}
\nabla_{E_j} E_1 = \frac{f}{\lambda}E_j, \quad E_j(\lambda)=0,
\end{equation}
for every $j=2,...,n$. Therefore the connection forms of $M$ are
\begin{eqnarray*}
\omega_1^i(E_1)&=&0, \quad \quad i \in \{1,...,n\}, \\ \omega_1^k(E_j)&=&\frac{f}{\lambda}\delta_{jk}, \quad j,k \in \{2,...,n\}.
\end{eqnarray*}
Using similar argumens as in the proof of Theorem \ref{nec-suf}, we may show that $M$ is locally the warped product $J\times_{\lambda} N$ such that the metric tensor of $M$ is given by 
$$
g=ds^2+\lambda^2g_N,
$$
where $\lambda$ is the warping function and $g_N$ is the metric tensor of $N$. Moreover, Eq. \eqref{anti-rec-e1} is now
$$
(-1+\lambda^2)f +\frac{d\lambda}{ds}=0.
$$
The proof is completed by a first integration.
\end{proof}
 
\begin{remark} \label{euc-ref}
Let $E$ be the radial vector field in $\e^m$. We know that $V=E/|E|$ is the anti-torqued vector field on $\e^m-\{0\}$. Accordingly, every rectifying submanifold $M$ in $\e^m$ with axis $V$ is a rectifying submanifold in the sense of \cite{crs0}, namely if $\Psi:M\to\e^m-\{0\}$ an isometric immersion of $M$ into $\e^m-\{0\}$, then  we get
$$
V|_M=V\circ\Psi=\frac{\Psi}{|\Psi|}
$$
and, by Definition \ref{def-rs},
$$
0=\langle \frac{\Psi(p)}{|\Psi(p)|}, \text{Im }h_p \rangle=\langle \Psi(p), \text{Im }h_p \rangle, \quad \forall p\in M.
$$
\end{remark}

With Remark \ref{euc-ref}, in the Euclidean setting, we are able to establish an example of rectifying submanifolds with anti-torqued axis. For this, recall that by a {\it spherical submanifold} of $\e^m$ we mean a submanifold contained in a hypersphere of $\e^m$.

\begin{example}
Let $\Psi:M\to\e^m-\{0\}$ be an isometric immersion of a Riemannian manifold $M$ into $\e^m$ given by
$$
\Psi(s,t_2,...,t_n)=\sqrt{1+s^2}\Omega(s,t_2,...,t_n),
$$
where $(s,t_2,...,t_n)$ is a local coordinate system on $M$ and $\Omega$ is the isometric immersion of a spherical submanifold into $\e^m-\{0\}$ whose metric tensor is given by
\begin{equation}\label{omeg}
g_{\Omega}=\frac{1}{(1+s^2)^2}(ds^2+s^2\sum_{i,j=2}^ng_{ij}(t_2,...,t_n)dt_idt_j).
\end{equation}
From \cite[Theorem 4.2]{crs0}, we know that $\Psi^{\top}\neq0\neq\Psi^{\perp}$ and the following relation is satisfied
$$
\langle \Psi(p), \text{Im }h_p \rangle=0, \quad \forall p\in M.
$$
This implies that $M$ is a proper rectifying submanifold of $\e^m$ whose axis is the anti-torqued vector field $V=E/|E|$. We will show that the metric tensor of $M$ is given as in Eq. \eqref{anti-rec-met}. Notice here that the conformal scalar $f$ of $V$ is $\frac{1}{|E|}$ and so the restriction $f|_M$ fulfills
$$
(f\circ\Psi)(s,t_2,...,t_n)=\frac{1}{|\Psi(s,t_2,...,t_n)|}=\frac{1}{\sqrt{1+s^2}},
$$
where $\langle \Omega, \Omega \rangle =1$ is used. The partial derivatives of $\Psi$ are
\begin{eqnarray*}
\frac{\partial \Psi}{\partial s}&=&\frac{s}{\sqrt{1+s^2}}\Omega+\sqrt{1+s^2}\frac{\partial \Omega}{\partial s}, \\
\frac{\partial \Psi}{\partial t_i}&=&\sqrt{1+s^2}\frac{\partial \Omega}{\partial t_i}.
\end{eqnarray*}
If we consider \eqref{omeg}, then the metric tensor of $M$ is given by
$$
g_{\Psi}=ds^2+\frac{s^2}{1+s^2}\sum_{i,j=2}^ng_{ij}(t_2,...,t_n)dt_idt_j,
$$
where the warping function  of $M$ is
$$
\lambda(s)=\frac{s}{\sqrt{1+s^2}}.
$$ 
Up to a suitable integration constant, the integral given in \eqref{anti-rec-met} is verified, namely
\begin{eqnarray*}
\tanh \int^s f(u)du &=&\tanh \int^s \frac{du}{\sqrt{1+u^2}} \\
&=& \tanh (\sinh^{-1}(s)) \\
&=&\frac{s}{\sqrt{1+s^2}}.
\end{eqnarray*}
\end{example}

\begin{remark}
W point out that it can be found many examples of spherical submanbifolds in $\e^m$ whose metric tensor is given by \eqref{omeg}. For example, introducing a new variable $\sigma=\tan^{-1}(s)$, one gets
$$
\frac{s}{1+s^2}=\frac{1}{2}\sin(2\sigma).
$$
If we consider it in \eqref{omeg}, we find
$$
g_{\Omega}=d\sigma^2+\frac{1}{4}\sin^2(2\sigma)\sum_{i,j=2}^ng_{ij}(t_2,...,t_n)dt_idt_j.
$$
When the dimension of the submanifold is $2$, it is the metric tensor of spherical coordinate system $(\sigma,t)$ on $\s^1(\frac{1}{2})$.
\end{remark}

\section{Final Remarks}

We notice that our examples are all in the Euclidean setting. This is because this model space of Riemannian geometry globally admits an anti-torqued vector field. Inspite, the spheres globally do not admit such vector fields (see \cite{id}) and it is an open problem if the hyperpolic space globally admits a torqued or an anti-torqued vector field. We currently process this problem in \cite{amc}. 

On the other hand, if the submanifold is a curve or a hypersurface, then the notion of rectifying submanifold can be generalized. In \cite{amc1}, we introduce a new concept which we call {\it pointwise rectifying submanifold} in Riemannian manifolds and study pointwise rectifying curves whose pointwise function is a constant and axis is an anti-torqued vector field.

\section*{Acknowledgements}
This study was supported by Scientific and Technological Research Council of Turkey (TUBITAK) under the Grant Number (123F451). The authors thank to TUBITAK for their supports.


\medskip

 \end{document}